\documentclass[a4paper,11pt]{amsart}
\usepackage{wrapfig}
\usepackage[dvips]{graphicx}
\usepackage{amsmath,amsthm,amssymb}
\theoremstyle{definition}
\newtheorem{thm}{Theorem}[section]
\newtheorem{Def}[thm]{Definition}
\newtheorem{pro}[thm]{Proposition}
\newtheorem{cor}[thm]{Corollary}
\newtheorem{lem}[thm]{Lemma}
\newtheorem{ex}[thm]{Example}
\newtheorem{rem}[thm]{Remark}
\theoremstyle{definition}

\title[Morita equivalent subalgebras]{Morita equivalent subalgebras of irrational 
rotation algebras and real quadratic fields}
\author{Norio Nawata}
\begin{document}
\maketitle
\begin{abstract}
In this paper, we determine the isomorphic classes of Morita equivalent subalgebras of 
irrational rotation algebras. It is based on the solution of the quadratic Diophantine 
equations. We determine the irrational rotation algebras that have locally trivial 
inclusions. We compute the index of the locally trivial inclusions of 
irrational rotation algebras.
\ \\
\ \\
Key words: Irrational rotation algebras, Morita equivalence, $C^*$-index theory, 
Real quadratic fields.
\ \\
Mathmatics Subject Classifications (2000): Primary 46L05, Secondary 11D09, 11R11.
\end{abstract}
\section{Introduction}
Let $\theta$ be an irrational number. An irrational rotation algebra $A_\theta$ is the 
universal $C^*$-algebra generated by two unitaries $u$, $v$, with the relation 
$uv=e^{2\pi i\theta}vu$. It is simple and has a unique normalized trace. 
They were classified up to $C^*$-isomorphism and Morita equivalence \cite{Pim1},\cite{Rie1}.

$C^*$-index theory in \cite{Wat} is a $C^*$-algebraic version of index theory for subfactors 
by V. F. R. Jones \cite{Jon}. Let $N\subseteq M$ be $II_1$-factors. If $M$ is a hyperfinite 
factor and Jones index $[M:N]$ is finite, then $N$ is also a hyperfinite factor. Hence, $N$ 
is isomorphic to $M$ as a von Neumann algebra. In $C^*$-index theory, there exist many 
non-isomorphic subalgebras that are of finite index. We need to consider isomorphic classes
of subalgebras. K. Kodaka studied endomorphisms of certain irrational rotation algebras 
in \cite{kod}. Since $A_\theta$ is simple, the ranges of endomorphisms are isomorphic 
subalgebras of $A_\theta$. In this paper, we extend his results and study the 
$C^*$-subalgebras of $A_\theta$, that are Morita equivalent to $A_\theta$. 

Throughout the paper, we assume that a subalgebra has a common unit. We shall sketch 
the content of each section in this paper.

In section \ref{sec:Morita}, we determine the isomorphic classes of Morita equivalent 
subalgebras of irrational rotation algebras. For example, Morita equivalent subalgebras of 
$A_{\frac{5+\sqrt{5}}{10}}$ are isomorphic to $A_{\frac{5+\sqrt{5}}{10}}$ 
or $A_{\frac{5+\sqrt{5}}{2}}$. Morita equivalent subalgebras of $A_{\frac{3+\sqrt{3}}{6}}$ 
are isomorphic to $A_{\frac{3+\sqrt{3}}{6}}$, $A_{\frac{3+\sqrt{3}}{3}}$, 
$A_{\frac{3+\sqrt{3}}{2}}$ or $A_{\sqrt{3}}$. It is based on the solution of the 
quadratic Diophantine equations. The isomorphic classes of Morita equivalent subalgebras 
of irrational rotation algebras are related to arithmetic properties of real quadratic fields. 
We show that a part of the decomposition of prime ideals in real quadratic fields is 
connected with the isomorphic classes of Morita equivalent subalgebras of irrational rotation 
algebras. We expect that there exists a connection with the real multiplication program by 
Y. Manin \cite{Man}.

In section \ref{sec:Cer}, we determine the irrational rotation algebras that have locally 
trivial inclusions, where an inclusion $B\subseteq A$ of $C^*$-algebras is called a 
locally trivial inclusion if there exist a projection $q$ in $A$ and an isomorphism 
$\varphi$ of $qAq$ onto $(1-q)A(1-q)$ such that $B=\{x+\varphi(x)\in A; x\in qAq \}$. 
If $A$ is simple, then $qAq$ and $A$ are Morita equivalent. Hence, a locally trivial inclusion 
is a construction of a Morita equivalent subalgebra.

In section \ref{sec:ind}, we show that the index of the locally trivial inclusions of 
irrational rotation algebras are four, which is the same value as the case of 
subfactors(if we consider the minimal index due to F. Hiai \cite{hia},).

\vspace{3mm}
\textbf{Acknowledgements.} 
The author wishes to express his deep gratitude to Professor Y. Watatani for many helpful 
suggestions and guidance and to Professor K. Kodaka for helpful comments. The author also 
thanks many people in Kyushu university for helpful discussion.

\section{Morita equivalent subalgebras}
\label{sec:Morita}
Let $\tau_\theta$ denote a unique normalized trace of $A_\theta$. The notion $A\cong B$ means 
that $A$ is isomorphic to $B$ as a $C^*$-algebra. Let $M_k(A_\theta)$ denote the algebra of 
all $k\times k$ matrices over $A_\theta$. We denote by $\tau_{k,\theta}$ the unique normalized 
trace on $M_k(A_\theta)$. We refer the reader to B. Blackadar \cite{Bla} and K. Davidson 
\cite{Dav} for the basic properties of $C^*$-algebras.

First, we consider the condition of $k\in\Bbb{N}$ and $\eta\in\Bbb{R}-\Bbb{Q}$ such that 
$M_k(A_\eta)$ is isomorphic to a subalgebra of $A_\theta$.
\begin{lem}
\label{pro:1.1}
If $M_k(A_\eta)$ is isomorphic to a subalgebra $B$ of $A_\theta$ with a common unit, 
then  there exists a natural number $n$ such that $M_k(A_\eta)\cong A_{n\theta}$.
\end{lem}
\begin{proof}
Let $\varphi :M_k(A_\eta)\rightarrow B$ be an isomorphism. Since the trace is unique and a 
subalgebra has a common unit,  
\[\tau_\theta(\varphi(x))=\tau_{k,\eta}(x) \qquad x\in M_k(A_\eta).\]
By \cite{Rie1}(Proposition 1.3), there exist an integer $l$ and a projection $q_1$ in 
$M_k(A_\eta)$ such that $\tau_{k,\eta}(q_1)=\eta+l$. Since $\varphi(q_1)$ is a projection 
in $A_\theta$,
\[\eta+l\in(\Bbb{Z}+\Bbb{Z}\theta)\cap[0,1].\]
There exist integers $m_0$,$m_1$ such that $\eta+l=m_0\theta+m_1$. Hence,
\[A_\eta\cong A_{\eta+l}\cong A_{m_0\theta+m_1} \cong A_{m_0\theta}\cong A_{|m_0|\theta}.\] 
Define $n:=|m_0|$.

Let $q_2$ and $q_3$ projections in $M_k(A_\eta)$ such that 
$\tau_{k,\eta}(q_2)=\frac{\eta+l}{k}=\frac{m_0\theta+m_1}{k}$ and 
$\tau_{k,\eta}(q_3)=\frac{1-\eta-l}{k}=\frac{1-m_0\theta-m_1}{k}$.

Since $\varphi(q_2)$ and $\varphi(q_3)$ are projections in $A_\theta$,
\[\frac{m_0\theta+m_1}{k} , \frac{1-m_0\theta-m_1}{k}\in(\Bbb{Z}+\Bbb{Z}\theta)\cap[0,1].\]
Hence, $\frac{m_1}{k}$ and $\frac{1-m_1}{k}$ are integers.
This implies $k=1$. Therefore, $M_k(A_\eta)\cong A_{n\theta}$.
\end{proof}
We shall consider Morita equivalent subalgebra of $A_\theta$. Let $GL(2, \Bbb{Z})$ 
denote the group of $2\times 2$ matrices with entires in $\Bbb{Z}$ and with determinant 
$\pm1$, and let $GL(2,\Bbb{Z})$ act on the set of irrational numbers by
\[\left(
    \begin{array}{cc}
     a & b \\
     c & d 
    \end{array}
    \right)\theta=\frac{a\theta+b}{c\theta+d}. \]
\begin{pro}\label{pro:n1.1}
A $C^*$-algebra $B$ is isomorphic to a subalgebra of $A_\theta$ with a common unit and is 
Morita equivalent to $A_\theta$ if and only if there exists $n\in\Bbb{N}$ and 
$g\in GL(2, \Bbb{Z})$ such that $n\theta=g\theta$ and $B\cong A_{n\theta}$. 
\end{pro}
\begin{proof}
We assume that $B$ is isomorphic to a subalgebra of $A_\theta$ with a common unit and 
is Morita equivalent to $A_\theta$. Since $B$ is Morita equivalent to $A_\theta$ and 
has a unit, there exist $k\in\Bbb{N}$ and $g_0\in GL(2,\Bbb{Z})$ such that 
$B\cong M_k(A_{g_0\theta})$ by \cite{Rie2}(Corollary 2.6). Hence, there exists a natural 
number $n$ such that $B\cong A_{n\theta}$ by Lemma \ref{pro:1.1}. Therefore, $k=1$ and 
$A_{n\theta}\cong A_{g_0\theta}$. There exists an integer $l$ such that $n\theta=g_0\theta+l$ 
or $n\theta=-g_0\theta+l$ by \cite{Rie1}(Thorem 2). Let $g_1, g_2\in GL(2,\Bbb{Z})$ be 
\[g_1=\left(
    \begin{array}{cc}
     1 & l \\
     0 & 1 
    \end{array}
    \right),\quad
    g_2=\left(
    \begin{array}{cc}
    -1 & 0 \\
     0 & 1 
    \end{array}
    \right).\]
Define $g:=g_1g_0$ or $g:=g_1g_2g_0$. Then $g\in GL(2, \Bbb{Z})$ and $n\theta=g\theta$. 
Consequently, there exist $n\in\Bbb{N}$ and $g\in GL(2, \Bbb{Z})$ such that $n\theta=g\theta$ 
and $B\cong A_{n\theta}$.

\vspace{3mm}
Conversely, we assume that there exist $n\in\Bbb{N}$ and $g\in GL(2, \Bbb{Z})$ such that 
$n\theta=g\theta$ and $B\cong A_{n\theta}$. We consider a subalgebra generated by $u^n$ and 
$v$. Since $u^nv=e^{2\pi in\theta}vu^n$, it is isomorphic to $A_{n\theta}$. Since 
$n\theta=g\theta$, it is Morita equivalent to $A_\theta$ by \cite{Rie1}(Theorem 4). 
Consequently, $B$ is isomorphic to a subalgebra of $A_\theta$ and is Morita equivalent to 
$A_\theta$.
\end{proof}
We study the isomorphic classes of subalgebras of $A_\theta$. We need the following well 
known fact.
\begin{lem}\label{pro:1.2}
Let $n_0$, $n_1$ be natural numbers. If $A_{n_0\theta}$ is isomorphic to $A_{n_1\theta}$, 
then $n_0=n_1$.
\end{lem}
We consider the case where $\theta$ is not a quadratic irrational number.
\begin{thm}\label{thm:1.1}
Let $\theta$ be an irrational number. Assume that $\theta$ is not a quadratic number. 
If $B$ is a subalgebra of $A_\theta$ with a common unit and is Morita equivalent to 
$A_\theta$, then $B$ is isomorphic to $A_\theta$.
\end{thm}
\begin{proof}
On the contrary, we assume that $A_\theta$ had a non-isomorphic Morita equivalent subalgebra. 
Then there exist $n\in\Bbb{N}$ and $g\in GL(2, \Bbb{Z})$ such that $n\theta=g\theta$ and 
$n\neq1$ by Proposition \ref{pro:n1.1} and Lemma \ref{pro:1.2}. There exist integers 
$a$,$b$,$c$,$d$ such that $ad-bc=\pm1$ and $n\theta=\frac{a\theta+b}{c\theta+d}$. Hence, 
we have $nc\theta^2+(dn-a)\theta-b=0$.
Since $\theta$ is not a quadratic irrational number and $n$ is a natural number, 
$c=0$, $dn-a=0, b=0$. By $ad-bc=\pm1$, $ad=\pm1$. Therefore, 
$(a,d)=(1,1), (1,-1), (-1,1), (-1,-1)$. Hence, $n=\pm1$. This is a contradiction.
\end{proof}
We consider the case where $\theta$ is a quadratic irrational number. We may assume that 
$\theta$ satisfies $k\theta^2+l\theta+m=0$ with a natural number $k$ and integers $l$,$m$ 
such that $gcd(k,l,m)=1$. The equation is uniquely determined. Let $D=l^2-4km$ be the 
discriminant of $\theta$.
\begin{lem}\label{lem:1.1}
Let $\theta$ be a quadratic irrational number with $k\theta^2+l\theta+m=0$ as above. 
If $B$ is a subalgebra of $A_\theta$ with a common unit and is Morita equivalent to 
$A_\theta$, then there exists a divisor $n$ of $k$ such that $B$ is isomorphic to 
$A_{n\theta}$.
\end{lem}
\begin{proof}
By Proposition \ref{pro:n1.1}, there exist $n\in\Bbb{N}$ and  $g\in GL(2, \Bbb{Z})$ such that 
$n\theta=g\theta$ and $B\cong A_{n\theta}$. There exist integers $a$,$b$,$c$,$d$ such that 
$ad-bc=\pm1$ and $n\theta=\frac{a\theta+b}{c\theta+d}$.
Hence, we have $nc\theta^2+(dn-a)\theta-b=0$.\ \\
In the case $\theta=\frac{-l+\sqrt{D}}{2k}$,
\[\frac{nc(l^2-2l\sqrt{D}+D)}{4k^2}+\frac{(dn-a)(-l+\sqrt{D})}{2k}-b=0. \]
\[(\frac{nc(l^2+D)-2kdnl+2kal}{4k^2}-b)+(\frac{-2lnc+2kdn-2ka}{4k^2})\sqrt{D}=0 .\]
Since $\sqrt{D}$ is an irrational number,
\[a=\frac{n}{k}(kd-lc), \]
\[b=\frac{nc(l^2+D)-2kdnl+2kal}{4k^2} .\]
Then
\[b=\frac{nc(l^2+D)-2kdnl+2nl(kd-lc)}{4k^2}=\frac{nc(D-l^2)}{4k^2}=-\frac{n}{k}mc .\]
Since $ad-bc=\pm1$,
\[\frac{n}{k}(kd^2-ldc)-(-\frac{n}{k}mc^2)=\pm1 ,\]
\[kd^2-ldc+mc^2=\pm\frac{k}{n} .\]
Consequently,
\[\frac{k}{n}\in \Bbb{Z} .\]
In the case $\theta=\frac{-l-\sqrt{D}}{2k}$, it is proved in the same way.
\end{proof}
\begin{thm}\label{lem:1.2}
Let $\theta$ be a quadratic irrational number with $k\theta^2+l\theta+m=0$ as above. 
Assume that $n$ is a divisor of $k$ and $k=\alpha n$ for a natural number $\alpha$. 
Then there exists a subalgebra $B$ of $A_\theta$ with a common unit such that $B$ is 
isomorphic to $A_{n\theta}$ and is Morita equivalent to $A_\theta$ if and only if 
either $nx^2-lxy+\alpha my^2=1$ or $nx^2-lxy+\alpha my^2=-1$ has solutions of integers $x$ 
and $y$.
\end{thm}
\begin{proof}
We assume that there exists a subalgebra $B$ of $A_\theta$ with a common unit such that 
$B$ is isomorphic to $A_{n\theta}$ and is Morita equivalent to $A_\theta$. Since $A_\theta$ 
and $A_{n\theta}$ are Morita equivalent, there exists a $g\in GL(2, \Bbb{Z})$ such that 
$n\theta=g\theta$ . There exist integers $a$,$b$,$c$,$d$ such that $ad-bc=\pm1$ and 
\[g=\left(
    \begin{array}{cc}
     a & b \\
     c & d 
    \end{array}
    \right). \]
In the case $\theta=\frac{-l+\sqrt{D}}{2k}$, by the proof of Lemma \ref{lem:1.1},
\[a=\frac{n}{k}(kd-lc)=\frac{1}{\alpha}(kd-lc)=nd-\frac{lc}{\alpha}\in\Bbb{Z}, \]
\[b=-\frac{n}{k}mc=-\frac{mc}{\alpha}\in\Bbb{Z} .\]
There exists an integer $t$ such that $c=t\alpha$. If not, it contradicts $gcd(k,l,m)=1$. 
\[ad-bc=\frac{n\alpha d^2-l\alpha dt}{\alpha}+\frac{m\alpha^2t^2}{\alpha}=nd^2-ldt+\alpha mt^2 .\]
Since $ad-bc=\pm1$, $(x,y)=(d,t)$ is a solution of $nx^2-lxy+\alpha my^2=1$ or 
$nx^2-lxy+\alpha my^2=-1$ .

In the case $\theta=\frac{-l-\sqrt{D}}{2k}$, it is proved in the same way.

\vspace{3mm}
Conversely, we assume that either $nx^2-lxy+\alpha my^2=1$ or $nx^2-lxy+\alpha my^2=-1$ has 
solutions of integers $x$ and $y$. Let $(d,t)$ be a solution of this equation. Define 
$a:=nd-lt$, $b:=-mt$, $c:=\alpha t$,  
\[g:=\left(
    \begin{array}{cc}
     a & b \\
     c & d 
    \end{array}
    \right) .\]
Then
\[ad-bc=nd^2-dlt+\alpha mt^2=nd^2-ldt+\alpha mt^2=\pm1 .\]
Therefore, $g\in GL(2,\Bbb{Z})$.
\[g\theta -n\theta=\frac{a\theta +b}{c\theta +d}-n\theta =\frac{a\theta +b-nc\theta^2 -dn\theta}{c\theta +d} \]
\[=\frac{(nd-lt)\theta -mt-kt\theta^2 -dn\theta}{c\theta +d}=\frac{-t(k\theta^2 +l\theta +m)}{c\theta +d}=0 .\]
Hence, $g\theta=n\theta$.
By Proposition \ref{pro:n1.1}, there exists a subalgebra $B$ of $A_\theta$ with a common unit 
such that $B$ is isomorphic to $A_{n\theta}$ and is Morita equivalent to $A_\theta$.
\end{proof}
There exists an algorithm of solving the quadratic Diophantine equations of the theorem 
above \cite{gau}. 
Hence, we can determine the isomorphic classes of Morita equivalent subalgebras of irrational 
rotation algebras. 

We shall show some examples.
\begin{ex}\label{ex:1.1}
Let $\theta$ be an algebraic integer of a real quadratic field.
If $B$ is a subalgebra of $A_\theta$ with a common unit and is Morita equivalent to 
$A_\theta$, then $B$ is isomorphic to $A_\theta$.
\end{ex}
\begin{proof}
An algebraic integer $\theta$ is a solution of monic equation. By Lemma \ref{lem:1.1}, 
$A_\theta$ does not have any non-isomorphic Morita equivalent subalgebras.
\end{proof}
\begin{ex}\label{ex:1.2}
Let $\theta=\frac{1}{\sqrt{p}}$ with $p$ is a prime number.
If $B$ is a subalgebra of $A_\theta$ with a common unit and is Morita equivalent to 
$A_\theta$, then $B$ is isomorphic to $A_\theta$ or $A_{p\theta}$.
\end{ex}
\begin{proof}
It is easy to see that $\theta$ is a solution of $p\theta^2-1=0$. By Lemma \ref{lem:1.1}, 
it is sufficient to consider only $A_{p\theta}$. An equation $px^2-y^2=-1$ has an integer 
solution $(0,1)$. Hence, by Theorem \ref{lem:1.2}, Morita equivalent subalgebras of 
$A_\theta$ are $A_\theta$ and $A_{p\theta}$.
\end{proof}
\begin{ex}\label{ex:1.5}
Let $\theta=\frac{5+\sqrt{5}}{10}$.
If $B$ is a subalgebra of $A_\theta$ with a common unit and is Morita equivalent to 
$A_\theta$, then $B$ is isomorphic to $A_\theta$ or $A_{5\theta}$.
\end{ex}
\begin{proof}
It is easy to see that $\theta$ is a solution of $5\theta^2-5\theta+1=0$. By Lemma 
\ref{lem:1.1}, it is sufficient to consider only $A_{5\theta}$. An equation 
$5x^2+5xy+y^2=1$ has an integer solution $(1,-1)$. Hence, by Theorem \ref{lem:1.2}, 
Morita equivalent subalgebras of $A_\theta$ are $A_\theta$ and $A_{5\theta}$.
\end{proof}
\begin{ex}\label{ex:1.3}
Let $\theta=\frac{3+\sqrt{3}}{6}$.
If $B$ is a subalgebra of $A_\theta$ with a common unit and is Morita equivalent to 
$A_\theta$, then $B$ is isomorphic to $A_\theta$, $A_{2\theta}$ $A_{3\theta}$ or 
$A_{6\theta}$.
\end{ex}
\begin{proof}
It is easy to see that $\theta$ is a solution of $6\theta^2-6\theta+1=0$. By Lemma 
\ref{lem:1.1}, it is sufficient to consider $A_{6\theta}$, $A_{3\theta}$, $A_{2\theta}$. 
An equation $6x^2+6xy+y^2=1$ has an integer solution $(0,1)$. An equation $3x^2+6xy+2y^2=-1$ 
has an integer solution $(1,-1)$. An equation $2x^2+6xy+3y^2=-1$ has an integer solution 
$(1,-1)$. Hence, by Theorem \ref{lem:1.2}, Morita equivalent subalgebras of $A_\theta$ are 
$A_\theta$, $A_{2\theta}$, $A_{3\theta}$, $A_{6\theta}$.
\end{proof}
\begin{ex}\label{ex:1.4}
Let $\theta=\frac{-5+\sqrt{65}}{10}$.
If $B$ is a subalgebra of $A_\theta$ with a common unit and is Morita equivalent to 
$A_\theta$, then $B$ is isomorphic to $A_\theta$.
\end{ex}
\begin{proof}
It is easy to see that $\theta$ is a solution of $5\theta^2+5\theta-2=0$. By Lemma 
\ref{lem:1.1}, it is sufficient to consider only $A_{5\theta}$. We shall show that 
$5x^2-5xy-2y^2=\pm1$ do not have any solutions of integers. We consider the left hand side of 
this equations modulo $5$, then
\[5x^2-5xy-2y^2\equiv0,2,3 \quad mod\,5 .\] 
Hence, 
\[5x^2-5xy-2y^2\not\equiv\pm1 \quad mod\,5.\] 
Therefore, this equations do not have solutions of integers. By Theorem \ref{lem:1.2}, 
$A_\theta$ does not have any non-isomorphic Morita equivalent subalgebras.
\end{proof}
We shall show the isomorphic classes of Morita equivalent subalgebras of irrational rotation 
algebras are related to arithmetic properties of real quadratic fields. We refer the reader 
to Y. Manin and A. Panchishkin \cite{ManP} for the basic properties of algebraic number theory.
\begin{cor}\label{thm:main}
Let $\theta$ be a quadratic irrational number with $p\theta^2+l\theta+m=0$. Assume that $p$ 
is a prime number and  $l$,$m$ are integers such that $gcd(p,l,m)=1$. We denote by 
$\Bbb{K}:=\Bbb{Q}(\theta)$ an algebraic field generated by $\theta$. The ring of integers of 
$\Bbb{K}$ is denoted by $\mathcal{O}_\Bbb{K}$. If $A_\theta$ has a non-isomorphic Morita 
equivalent subalgebra, then $p\mathcal{O}_\Bbb{K}$ is not a prime ideal in 
$\mathcal{O}_\Bbb{K}$, that is, $p$ splits completely or is ramified in $\Bbb{K}$.
\end{cor}
\begin{proof}
Since $A_\theta$ has a non-isomorphic Morita equivalent subalgebra, $A_{p\theta}$ is 
isomorphic to a Morita equivalent subalgebra of $A_\theta$ by Lemma \ref{lem:1.1}. 
By Theorem \ref{lem:1.2}, either $px^2-lxy+my^2=1$ or $px^2-lxy+my^2=-1$ has solutions of 
integers. We denote by $(d,t)$ one of them. Compute this equation,
\[(pd-\frac{l}{2}t+\frac{\sqrt{D}}{2}t)(pd-\frac{l}{2}t-\frac{\sqrt{D}}{2}t)=\pm p \]
Because $\frac{-l+\sqrt{D}}{2}$ and $\frac{-l-\sqrt{D}}{2}$ are solutions of 
$\eta^2+l\eta+pm=0$, 
\[pd-\frac{l}{2}t+\frac{\sqrt{D}}{2}t,\quad pd-\frac{l}{2}t-\frac{\sqrt{D}}{2}t\in\mathcal{O}_\Bbb{K}. \] 
It is easy to see that
\[pd-\frac{l}{2}t+\frac{\sqrt{D}}{2}t,\quad pd-\frac{l}{2}t-\frac{\sqrt{D}}{2}t\notin p\mathcal{O}_\Bbb{K}. \]
Hence, $p\mathcal{O}_\Bbb{K}$ is not a prime ideal in $\mathcal{O}_\Bbb{K}$, that is, 
$p\mathcal{O}_\Bbb{K}$ splits completely or is ramified.
\end{proof}
\section{Locally trivial inclusions}
\label{sec:Cer}
In section \ref{sec:Morita}, we consider a subalgebra generated by $u^n$, $v$ and the 
isomorphic classes of the subalgebra. In this section, we shall show there exist other 
Morita equivalent subalgebras of certain irrational rotation algebras. First, we define 
a locally trivial inclusion.
\begin{Def}
Let $B\subseteq A$ be $C^*$-algebras. An inclusion $B\subseteq A$ is called a locally trivial 
inclusion if there exist a projection $q$ in $A$ and an isomorphism $\varphi$ of $qAq$ 
onto $(1-q)A(1-q)$ such that $B=\{x+\varphi(x)\in A; x\in qAq \}$.
\end{Def}
Since $B$ is isomorphic to $qAq$, if $A$ is simple, then $B$ is Morita equivalent to $A$.

K. Kodaka determined the irrational rotation algebras that have a locally trivial inclusion 
$B\subseteq A_\theta$ with $B\cong A_\theta$ in \cite{kod}. He showed that there exists a 
projection $q$ in $A_\theta$ such that $A_\theta\cong qA_\theta q\cong (1-q)A_\theta(1-q)$ 
if and only if the discriminant of $\theta$ is five. We determine the irrational rotation 
algebras that have a locally trivial inclusion $B\subseteq A_\theta$. We do not assume that 
$B$ is isomorphic to $A_\theta$.

Let $c$, $d$ be integers with $gcd(c,d)=1$. We also assume $c\neq 0$. Let $V_\theta(d,c;k)$ be 
the standard module defined in \cite{Rie2} where $k$ is a natural number. It is a 
$M_k(A_\frac{a\theta+b}{c\theta+d})-A_\theta$-equivalence bimodule constructed in \cite{Rie2} 
for any integers $a$, $b$ such that $ad-bc=\pm1$. Since $V_\theta(d,c;k)$ is a finitely 
generated projective right $A_\theta$-module, it corresponds to a projection in some 
$M_k(A_\theta)$. We also denote it by $V_\theta(m,l;k)$. Let 
$Tr_\theta :=\tau_\theta\otimes Tr$ be the unnormalized trace on $M_k(A_\theta)$ where 
$Tr$ is the usual trace on $M_k(\Bbb{C})$. The following lemma is based on a proof in 
K. Kodaka \cite{kod}(Lemma 7).
\begin{lem}\label{lem:kod}
If $q$ is a proper projection in $A_\theta$ such that $\tau_\theta(q)=k(c\theta+d)$ where
 $k$ is a natural number and $c$, $d$ are integers such that $gcd(c,d)=1$, then
\[qA_\theta q\cong M_k(A_{\frac{a\theta +b}{c\theta +d}}) \]
for any $a, b\in \Bbb{Z}$ such that $ad-bc=\pm1$.
\end{lem}
\begin{proof}
By \cite{Rie2}(Theorem 1.4), $Tr_\theta(V_\theta(d,c;k))=k(c\theta+d)$. Let $e_1$ be a rank 
one projection in $M_k(\Bbb{C})$. Then $qA_\theta$ is isomorphic to 
$(q\otimes e_1)M_k(A_\theta)$ as a module. 
Since $Tr_\theta(q\otimes e_1)=k(c\theta+d)=Tr_\theta(V_\theta(c,d;k))$, $qA_\theta$ is 
isomorphic to $V_\theta(c,d;k)$ as a module by \cite{kod}(Lemma 6). Since $qA_\theta$ is 
the $qA_\theta q$-$A_\theta$-equivalence bimodule, 
\[qA_\theta q\cong M_k(A_{\frac{a\theta +b}{c\theta +d}}). \]
for any $a, b\in \Bbb{Z}$ such that $ad-bc=\pm1$ by \cite{Rie2}(Theorem 1.1 and Corollary 2.6).
\end{proof}
Let\ \\
$S_1=\{\frac{-K(2d-1)\pm\sqrt{K^2-4K}}{2cK};K,c,d\in\Bbb{Z},K\geq 5, gcd(c,d)=1,\frac{Kd^2-Kd+1}{c}\in \Bbb{Z} \}$,
\ \\
\ \\
$S_2=\{\frac{2-K(2d-1)-\sqrt{K^2+4}}{2cK};K,c,d\in\Bbb{Z},K\neq 0, gcd(c,d)=1,\frac{Kd^2-Kd-2d+1}{c}\in \Bbb{Z} \}$.\ \\ \ \\
We consider the condition of an irrational rotation algebra $A_\theta$ that has a locally 
trivial inclusion.
\begin{lem}\label{lem:con}
If an irrational rotation algebra $A_\theta$ has a locally trivial inclusion, then 
$\theta\in S_1\cup S_2$.
\end{lem}
\begin{proof}
There exists a projection $q$ such that $qA_\theta q\cong (1-q)A_\theta (1-q)$. By 
\cite{Rie1}(Proposition 1.3), there exist integers $c$, $d$ and a natural number $k$ such that 
$\tau_\theta(q)=k(c\theta +d)$ and $gcd(c,d)=1$. Since 
$qA_\theta q\cong (1-q)A_\theta (1-q)$, $c\neq 0$. By Lemma \ref{lem:kod},
\[qA_\theta q\cong M_k(A_{\frac{a\theta +b}{c\theta +d}}) \]
for any $a, b\in \Bbb{Z}$ such that $ad-bc=\pm1$. Fix $a, b\in \Bbb{Z}$ such that $ad-bc=1$. 
Since $M_k(A_{\frac{a\theta +b}{c\theta +d}})$ is isomorphic to a subalgebra of $A_\theta$ 
with a common unit, $k=1$ by Lemma \ref{pro:1.1}. By Lemma \ref{lem:kod} and 
$qA_\theta q\cong (1-q)A_\theta (1-q)$, there exist integers $s$, $t$ such that 
$s(1-d)+tc=\pm 1$ and
\[\frac{a\theta +b}{c\theta +d}=\frac{s\theta +t}{-c\theta +1-d}. \]
Compute the equation above,
\[(s+a)c\theta^2+((s+a)d+(t+b)c-a)\theta +(t+b)d-b=0. \]
(i) The case $s(1-d)+tc=1$.\ \\
By $ad-bc=1$ and $s(1-d)+tc=1$, $(t+b)c=(s+a)d-s$. Since 
$c\neq 0$, $t+b=\frac{(s+a)d-s}{c}\in \Bbb{Z}$. Hence,
\[(s+a)c\theta^2+(2(s+a)d-(s+a))\theta +\frac{(s+a)d^2-(s+a)d+1}{c}=0. \]
Since $\frac{(s+a)d-s}{c}\in \Bbb{Z}$ and $-bc=-ad+1$, 
\[\frac{(s+a)d^2-(s+a)d+1}{c}=d\frac{(s+a)d-s}{c}-b\in \Bbb{Z}. \] 
Define $K:=s+a$. Then
\[Kc\theta^2+(2Kd-K)\theta +\frac{Kd^2-Kd+1}{c}=0 \]
and
\[\frac{Kd^2-Kd+1}{c}\in \Bbb{Z}. \]
Solve the equation above, then
$\theta=\frac{-K(2d-1)\pm\sqrt{K^2-4K}}{2cK}$.
Since $0<\tau_\theta(q)<1$,
$0<c\theta +d=\frac{K\pm\sqrt{K^2-4K}}{2K}<1$.
Hence, $K\geq 5$. Consequently, $\theta\in S_1$.\ \\
\ \\
(ii) The case $s(1-d)+tc=-1$.\ \\
By $ad-bc=1$ and $s(1-d)+tc=-1$, $(t+b)c=(s+a)d-s-2$. Since 
$c\neq 0$, $t+b=\frac{(s+a)d-s-2}{c}\in \Bbb{Z}$. Hence,
\[(s+a)c\theta^2+(2(s+a)d-(s+a)-2)\theta +\frac{(s+a)d^2-(s+a)d-2d+1}{c}=0. \]
Since $\frac{(s+a)d-s-2}{c}\in \Bbb{Z}$ and $-bc=-ad+1$,
\[\frac{(s+a)d^2-(s+a)d-2d+1}{c}=d\frac{(s+a)d-s-2}{c}-b\in \Bbb{Z}. \]
Define $K:=s+a$. Then
\[Kc\theta^2+(2Kd-K-2)\theta +\frac{Kd^2-Kd-2d+1}{c}=0 \]
and
\[\frac{Kd^2-Kd-2d+1}{c}\in \Bbb{Z}. \]
Solve an equation above, then
$\theta=\frac{-K(2d-1)+2\pm\sqrt{K^2+4}}{2cK}$.
Since $0<\tau_\theta(q)<1$,
$0<\frac{K+2\pm\sqrt{K^2+4}}{2K}<1$.
Hence, 
\[\theta=\frac{-K(2d-1)+2-\sqrt{K^2+4}}{2cK}. \]
Consequently, $\theta\in S_2$.
\end{proof}
\begin{lem}\label{lem:s.1}
If $\theta\in S_1$, then $A_\theta$ has a locally trivial inclusion.
\end{lem}
\begin{proof}
There exist integers $c$, $d$, and $K$ such that $gcd(c,d)=1$, $K\geq 5$, 
$\frac{Kd^2-Kd+1}{c}\in \Bbb{Z}$ and
$\theta=\frac{-K(2d-1)\pm\sqrt{K^2-4K}}{2cK}$.
Since $K\geq 5$,
$0<c\theta +d=\frac{K\pm\sqrt{K^2-4K}}{2K}<1$.
By \cite{Rie1}(Proposition1.3), there exists a projection $q$ in $A_\theta$ such that 
$\tau_\theta (q)=c\theta +d$. By $gcd(c,d)=1$, there exist integers $a$, $b$ such that 
$ad-bc=1$. Define $s:=K-a$ and $t:=\frac{s(d-1)+1}{c}$. We shall show that $t$ is an 
integer.
\[t=\frac{(K-a)(d-1)+1}{c}=\frac{Kd-K-ad+a+1}{c} \]
\[=\frac{Kd-K-bc-1+a+1}{c}=-b+\frac{Kd-K+a}{c} \]
Hence, we only need to show that $\frac{Kd-K+a}{c}$ is an integer. Since $gcd(c,d)=1$, 
it is sufficient to show that $\frac{d}{c}(Kd-K+a)$ is an integer.
\[\frac{d}{c}(Kd-K+a)=\frac{Kd^2-Kd+ad}{c} \]
\[=\frac{Kd^2-Kd+bc+1}{c}=b+\frac{Kd^2-Kd+1}{c} \]
Since $\frac{Kd^2-Kd+1}{c}\in \Bbb{Z}$, $\frac{d}{c}(Kd-K+a)$ is an integer. Therefore, 
$t$ is an integer. It is easy to see that $s(1-d)+tc=1$. By Lemma\ref{lem:kod},
\[qA_\theta q\cong A_\frac{a\theta +b}{c\theta +d}, \qquad (1-q)A_\theta (1-q)\cong A_\frac{s\theta +t}{-c\theta +1-d}. \] 
By computation,
\[\frac{a\theta +b}{c\theta +d}-\frac{s\theta +t}{-c\theta +1-d}=-\frac{Kc\theta^2+(2Kd-K)\theta +\frac{Kd^2-Kd+1}{c}}{(c\theta +d)(-c\theta +1-d)}=0. \]
Hence,
$\frac{a\theta +b}{c\theta +d}=\frac{s\theta +t}{-c\theta +1-d}$.
Therefore, $qA_\theta q\cong (1-q)A_\theta (1-q)$. Consequently, $A_\theta$ has a locally 
trivial inclusion.
\end{proof}
\begin{lem}\label{lem:s.2}
If $\theta\in S_2$, then $A_\theta$ has a locally trivial inclusion.
\end{lem}
\begin{proof}
There exist integers $c$, $d$, and $K$ such that $gcd(c,d)=1$, $K\neq 0$, 
$\frac{Kd^2-Kd-2d+1}{c}\in \Bbb{Z}$ and
$\theta=\frac{-K(2d-1)+2-\sqrt{K^2+4}}{2cK}$.
It is easy to see that
$0<c\theta +d=\frac{K+2-\sqrt{K^2+4}}{2K}<1$.
By \cite{Rie1}(Proposition 1.3), there exists a projection $q$ in $A_\theta$ such that 
$\tau_\theta (q)=c\theta +d$. By $gcd(c,d)=1$, there exist integers $a$, $b$ such that 
$ad-bc=1$. Define $s:=K-a$ and $t:=\frac{s(d-1)-1}{c}$. We shall show that $t$ is an integer.
\[t=\frac{(K-a)(d-1)-1}{c}=\frac{Kd-K-ad+a-1}{c} \]
\[=\frac{Kd-K-bc-1+a-1}{c}=-b+\frac{Kd-K+a-2}{c} \]
Hence, we only need to show that $\frac{Kd-K+a-2}{c}$ is an integer. Since $gcd(c,d)=1$, 
it is sufficient to show that $\frac{d}{c}(Kd-K+a-2)$ is an integer.
\[\frac{d}{c}(Kd-K+a-2)=\frac{Kd^2-Kd+ad-2d}{c} \]
\[=\frac{Kd^2-Kd+bc+1-2d}{c}=b+\frac{Kd^2-Kd-2d+1}{c} \]
Since $\frac{Kd^2-Kd-2d+1}{c}\in \Bbb{Z}$, $\frac{d}{c}(Kd-K+a-2)$ is an integer. 
Therefore, $t$ is an integer. It is easy to see that $s(1-d)+tc=-1$. By Lemma \ref{lem:kod},
\[qA_\theta q\cong A_\frac{a\theta +b}{c\theta +d}, \qquad (1-q)A_\theta (1-q)\cong A_\frac{s\theta +t}{-c\theta +1-d}. \] 
By computation,
\[\frac{a\theta +b}{c\theta +d}-\frac{s\theta +t}{-c\theta +1-d}=-\frac{Kc\theta^2+(2Kd-K-2)\theta +\frac{Kd^2-Kd-2d+1}{c}}{(c\theta +d)(-c\theta +1-d)}=0. \]
Hence,
$\frac{a\theta +b}{c\theta +d}=\frac{s\theta +t}{-c\theta +1-d}$.
Therefore, $qA_\theta q\cong (1-q)A_\theta (1-q)$. Consequently, $A_\theta$ has a locally 
trivial inclusion.
\end{proof}
We shall determine the locally inclusions of irrational rotation algebra.
\begin{thm}We have the following.\ \\ \ \\
(1) Let $\theta=\frac{-K(2d-1)+\sqrt{K^2-4K}}{2cK}$ with $K,c,d\in\Bbb{Z}$, $K\geq 5$, $gcd(c,d)=1$ and $\frac{Kd^2-Kd+1}{c}\in \Bbb{Z}$. Then the irrational rotation algebra $A_\theta$ has a locally trivial inclusion $A_{K\theta}\subseteq A_\theta$.\ \\
\ \\
(2) Let $\theta=\frac{-K(2d-1)-\sqrt{K^2-4K}}{2cK}$ with $K,c,d\in\Bbb{Z}$, $K\geq 5$, $gcd(c,d)=1$ and $\frac{Kd^2-Kd+1}{c}\in \Bbb{Z}$. Then the irrational rotation algebra $A_\theta$ has a locally trivial inclusion $A_{K\theta}\subseteq A_\theta$.\ \\
\ \\
(3) Let $\theta=\frac{-K(2d-1)+2-\sqrt{K^2+4}}{2cK}$ with $K,c,d\in\Bbb{Z}$, $K\neq 0$, $gcd(c,d)=1$ and $\frac{Kd^2-Kd-2d+1}{c}\in \Bbb{Z}$. Then the irrational rotation algebra $A_\theta$ has a locally trivial inclusion $A_{K\theta}\subseteq A_\theta$.\ \\
\ \\
(4) Let $\theta$ be an irrational number and not in (1), (2) and (3). Then the irrational rotation algebra $A_\theta$ does not have any locally trivial inclusions.
\end{thm}
\begin{proof}\ \\
Proof of (1). By Lemma \ref{lem:s.1}, $A_\theta$ has a locally trivial inclusion 
$A_{\frac{a\theta +b}{c\theta +d}}\subseteq A_\theta$ for any $a, b\in \Bbb{Z}$ such that 
$ad-bc=\pm1$. Fix $a, b\in \Bbb{Z}$ such that $ad-bc=1$. Define 
\[m:=\frac{Kd-K+a}{c}. \]
By the proof of Lemma \ref{lem:s.1}, we have $m\in \Bbb{Z}$. Simple computation shows that
\[\frac{a\theta +b}{c\theta +d}-m=\frac{-K(2d-1)+\sqrt{K^2-4K}}{2c}=K\theta. \]
Consequently, $A_{\frac{a\theta +b}{c\theta +d}}\cong A_{K\theta}$.
\ \\
\ \\
Proof of (2). By a similar argument of (1), it is proved.
\ \\
\ \\
Proof of (3). By Lemma \ref{lem:s.2}, $A_\theta$ has a locally trivial inclusion 
$A_{\frac{a\theta +b}{c\theta +d}}\subseteq A_\theta$ for any $a, b\in \Bbb{Z}$ such that 
$ad-bc=\pm1$. Fix $a, b\in \Bbb{Z}$ such that $ad-bc=1$. Define
\[m:=\frac{Kd-K+a-2}{c}. \]
By the proof of Lemma \ref{lem:s.2}, we have $m\in \Bbb{Z}$. Simple computation shows that
\[\frac{a\theta +b}{c\theta +d}-m=\frac{-K(2d-1)+2-\sqrt{K^2+4}}{2c}=K\theta \]
Consequently, $A_{\frac{a\theta +b}{c\theta +d}}\cong A_{K\theta}$.
\ \\
\ \\
Proof of (4). This is immediate by Lemma \ref{lem:con}.
\end{proof}
The case where K. Kodaka studied in \cite{kod} is (3) with $K=\pm1$. We shall show some 
examples.
\begin{ex}
Let $\theta$ be an algebraic integer of a real quadratic field. If the discriminant of 
$\theta$ is not five, then $A_\theta$ does not have any locally trivial inclusions.
\end{ex}
\begin{ex}
Let $\theta =\frac{5+\sqrt{5}}{10}$. Then $A_\theta$ has locally trivial inclusions 
$A_\theta\subseteq A_\theta$ and $A_{5\theta}\subseteq A_\theta$.
\end{ex}
\begin{ex}
Let $\theta =\frac{3+\sqrt{3}}{6}$. Then $A_\theta$ has a locally trivial inclusion 
$A_{6\theta}\subseteq A_\theta$.
\end{ex}

\section{The index of the locally trivial inclusions}\label{sec:ind}
In this section, we compute the index of the locally trivial inclusions of irrational rotation 
algebras. First, we review some definitions of $C^*$-index theory in \cite{Wat}.
\begin{Def}
Let $B\subseteq A$ be $C^*$-algebras with a common unit and $E:A\rightarrow B$ a 
conditional expectation. A finite family $\{(u_1, v_1),..., (u_n, v_n)\}$ is called 
quasi-basis if the following equations hold:
\[x=\sum_iu_iE(v_ix)=\sum_iE(xu_i)v_i \qquad x\in A .\]
We say that a conditional expectation $E:A\rightarrow B$ is of index-finite type if there 
exists a quasi-basis for $E$. In this case we define the index of $E$ by
\[IndexE:=\sum_iu_iv_i .\]
\end{Def}
$IndexE$ belongs to the centre of $A$ and does not depend on the choice of quasi-basis.
\begin{Def}
Assume that the centre of $A$ and the centre of $B$ are scalars. Let $\epsilon_0(A,B)$ 
denote the set of all expectations of $A$ onto $B$ of index-finite type. We define minimal 
index $[A:B]_0$ by
\[[A:B]_0:=min\{IndexF ; F\in\epsilon_0(A,B)\} .\]
\end{Def}
Let $B$ be a subalgebra of $A_\theta$ generated by $u^n$ and $v$. Then 
$[A_\theta :B]_0=n$ \cite{Wat}.

Throughout this section, we assume that $q$ is a projection in $A_\theta$ such that 
$qA_\theta q\cong (1-q)A_\theta (1-q)$. Let $\varphi$ denote an isomorphism of $qA_\theta q$ 
onto $(1-q)A_\theta (1-q)$. We may assume $\tau_\theta(q)>1/2$. Let 
$B:=\{x+\varphi(x); x\in qA_\theta q\}$.  

We define a conditional expectation $E:A_\theta\rightarrow B$ by,
\[E(x):=\frac{1}{2}(qxq+(1-q)x(1-q)+\varphi(qxq)+\varphi^{-1}((1-q)x(1-q))). \]
Then by an easy computation, we see that $E$ is a faithful conditional expectation.

The following lemma is well known.
\begin{lem}\label{lem:4.1}
Let $q_1$, $q_2$ be projections in $A_\theta$. If $\tau_\theta(q_1)\geq\tau_\theta(q_2)$, 
then there exists a unitary element $w$ in $A_\theta$ such that $q_1\geq w^*q_2w$.
\end{lem}
We shall show the key lemma.
\begin{lem}\label{lem:4.2}
There exist a natural number $n$ and a projection $q_0$ in $(1-q)A_\theta (1-q)$ and orthogonal 
projections $q_1,...,q_n$ in $qA_\theta q$ and unitary elements $w_1,...,w_n$ in $A_\theta$ 
such that $q=q_1+q_2+...+q_n$ and $q_i=w_i^*(1-q)w_i, (0<i<n)$ and $q_n=w_n^*q_0w_n$.  
\end{lem}
\begin{proof}
Since we assume $\tau_\theta(q)>1/2$, $\tau_\theta(q)>\tau_\theta(1-q)$. It is easy to see that 
there exists a unique natural number $k$ such that 
\[0<\tau_\theta(q)-k\tau_\theta(1-q)<\tau_\theta(1-q). \]
Define $n:=k+1$.

By Lemma \ref{lem:4.1}, there exist a projection $q_1$ in $qA_\theta q$ and a unitary element 
$w_1$ in $A_\theta$ such that $q_1=w_1^*(1-q)w_1$. We shall do the same way in $(1<i<n)$. 
Since
\[\tau_\theta(1-q)<\tau_\theta(q)-(i-1)\tau_\theta(1-q)=\tau_\theta(q-q_1-...-q_{i-1}), \]
there exist a projection $q_i$ in $(q-q_1-...-q_{i-1})A_\theta (q-q_1-...-q_{i-1})$ and a 
unitary element $w_i$ in $A_\theta$ such that $q_i=w_i^*(1-q)w_i$.
Define $q_n:=q-q_1-...-q_{n-1}$. Since $\tau_\theta(1-q)>\tau_\theta(q_n)$, there exist a 
projection $q_0$ in $(1-q)A_\theta(1-q)$ and a unitary element $w_n$ such that 
$q_n=w_n^*q_0w_n$. Therefore, we obtain the conclusion.
\end{proof}
We shall construct a quasi-basis for $E$ by the lemma above.
\begin{lem}
A family $\{(u_1, u_1^*), ..., (u_{n+3}, u_{n+3}^*)\}:=$
\[\left\{
    \begin{array}{lcr}
     (\sqrt{2}\,q,\sqrt{2}\,q), (\sqrt{2}\,(1-q),\sqrt{2}\,(1-q)), (\sqrt{2}\,(1-q)w_1q,\sqrt{2}\,qw_1^*(1-q)), \\
     (\sqrt{2}\,w_1^*(1-q),\sqrt{2}\,(1-q)w_1), ....... , (\sqrt{2}\,w_{n-1}^*(1-q),\sqrt{2}\,(1-q)w_{n-1}), \\
      (\sqrt{2}\,w_n^*q_0,\sqrt{2}\,q_0w_n)
    \end{array}
    \right\} \]
is a quasi-basis for $E$.
\end{lem}
\begin{proof}
It is sufficient to show that
\[x=\sum_iu_iE(u_i^*x) \qquad x\in A_\theta. \]
For any $x\in A_\theta$, we see by Lemma \ref{lem:4.2},
\[qE(qx)=\frac{1}{2}qxq, \]
\[(1-q)E((1-q)x)=\frac{1}{2}(1-q)x(1-q), \]
\[(1-q)w_1qE(qw_1^*(1-q)x)=\frac{1}{2}(1-q)xq, \]
$0<i<n$
\[w_i^*(1-q)E((1-q)w_ix)=\frac{1}{2}q_ix(1-q), \]
\[w_n^*q_0E(q_0w_nx)=\frac{1}{2}q_nx(1-q). \]
Therefore,
\begin{eqnarray*}
x &=& qxq+(1-q)x(1-q)+(1-q)xq  \\
  & & {}+q_1x(1-q)+.....+q_{n-1}x(1-q)+q_nx(1-q) \\
  &=& \sqrt{2}\,qE(\sqrt{2}\,qx)+\sqrt{2}\,(1-q)E(\sqrt{2}\,(1-q)x) \\
  & & {}+\sqrt{2}\,(1-q)w_1qE(\sqrt{2}\,qw_1^*(1-q)x) \\
  & & {}+\sqrt{2}\,w_1^*(1-q)E(\sqrt{2}\,(1-q)w_1x)+..... \\
  & & {}+\sqrt{2}\,w_{n-1}^*(1-q)E(\sqrt{2}\,(1-q)w_{n-1}x) \\
  & & {}+\sqrt{2}\,w_n^*q_0E(\sqrt{2}\,q_0w_nx) \\
\end{eqnarray*}
Consequently, this family is a quasi-basis.
\end{proof}
\begin{lem}\label{cor:1.1}
\[IndexE=4. \]
\end{lem}
\begin{proof}
\begin{eqnarray*}
IndexE &=& \sqrt{2}\,q\sqrt{2}\,q+\sqrt{2}\,(1-q)\sqrt{2}\,(1-q) \\
       & & {}+\sqrt{2}\,(1-q)w_1q\sqrt{2}\,qw_1^*(1-q) \\
       & & {}+\sqrt{2}\,w_1^*(1-q)\sqrt{2}\,(1-q)w_1+..... \\
       & & {}+\sqrt{2}\,w_{n-1}^*(1-q)\sqrt{2}\,(1-q)w_{n-1} \\
       & & {}+\sqrt{2}\,w_n^*q_0\sqrt{2}\,q_0w_n \\
       &=& 2q+2(1-q)+2(1-q)+2q_1+....+2q_n \\
       &=& 4
\end{eqnarray*}
\end{proof}
\begin{thm}\label{cor:1.2}
Let $B\subseteq A_\theta$ be a locally trivial inclusion. Then
\[[A_\theta:B]_0=4. \]
\end{thm}
\begin{proof}
This is immediate by Lemma \ref{cor:1.1} and \cite{Wat}(Theorem 2,12.3).
\end{proof}
\begin{rem}(1) The index of the locally trivial inclusions of irrational rotation algebras is the same 
value as the case of subfactors(if we consider the minimal index due to F. Hiai \cite{hia},).
\ \\
(2) Since simple TAI $C^*$-algebras which have a unique trace have the property of Lemma \ref{lem:4.1} 
(see Corollary 4.6 and Theorem 4.7 in \cite{Lin}), we can show that the index of the locally 
trivial inclusions of simple TAI $C^*$-algebras which have a unique trace is four in the similar way.  
\end{rem}

(Norio Nawata) G{\scriptsize RADUATE} S{\scriptsize CHOOL OF} M{\scriptsize ATHEMATICS}, K{\scriptsize YUSHU} U{\scriptsize NIVERSITY}, H{\scriptsize AKOZAKI}, F{\scriptsize UKUOKA}, 812-8581, J{\scriptsize APAN}\ \\
E-$mail$ $address$: n-nawata@math.kyushu-u.ac.jp

\end{document}